\documentclass[a4paper,11pt,DIV=11,
abstract=on
]{scrartcl}
\usepackage{algorithm, booktabs}
\usepackage{mathtools}
\mathtoolsset{showonlyrefs}
\usepackage{amsfonts}
\usepackage{amsmath}
\usepackage{amssymb}
\usepackage{faktor}
\usepackage{amscd}
\usepackage{array}
\usepackage{amsthm}
\usepackage{subfigure}
\usepackage{mathdots}
\usepackage{mathrsfs}
\usepackage{algorithm}
\usepackage[noend]{algpseudocode}
\usepackage{graphicx}
\usepackage{color}
\usepackage{pgf}
\usepackage{scrlayer-scrpage}
\usepackage{multirow}
\usepackage{listings}
\usepackage{tikz}
\usepackage{graphicx}
\usepackage[left=2.5cm, right=2.5cm,top=2.5cm,bottom=2.5cm]{geometry}
\usepackage{scalerel}

\newcommand{\weakly}{\rightharpoonup}
\DeclareMathOperator{\Id}{I}
\DeclareMathOperator{\Fix}{Fix}

\newtheorem{thm}{Theorem}[section]
\newtheorem{lemma}[thm]{Lemma}
\newtheorem{remark}[thm]{Remark}

\newtheorem{example}[thm]{Example}
\newtheorem{corollary}[thm]{Corollary}
\newtheorem{proposition}[thm]{Proposition}

\newcommand{\N}{\ensuremath{\mathbb{N}}}

\newcommand{\R}{\ensuremath{\mathbb{R}}}

\begin{document}

	\title{On $\alpha$-Firmly Nonexpansive Operators in $r$-Uniformly Convex  Spaces}
	\author{Arian B\"erd\"ellima \footnotemark[1] 
\and
 Gabriele Steidl\footnotemark[1] 
}
\maketitle
\date{\today}

\footnotetext[1]{Institute of Mathematics,
	TU Berlin,
	Stra{\ss}e des 17. Juni 136, 
	D-10623 Berlin, Germany,
	\{berdellima,steidl\}@math.tu-berlin.de
	}

	\begin{abstract}
		We introduce the class of $\alpha$-firmly nonexpansive and 
		quasi $\alpha$-firmly nonexpansive operators 
		on $r$-uniformly convex Banach spaces. 
		This extends the existing notion from Hilbert spaces, where
		$\alpha$-firmly nonexpansive operators coincide with so-called $\alpha$-averaged operators.  
		For our more general setting, we show that $\alpha$-averaged operators form a subset of
		$\alpha$-firmly nonexpansive operators.
		We develop some basic calculus rules for (quasi) $\alpha$-firmly nonexpansive operators. 
		In particular, we show that their compositions and convex combinations are again (quasi) $\alpha$-firmly nonexpansive.
		Moreover, we will see that quasi $\alpha$-firmly nonexpansive operators enjoy the asymptotic regularity property. Then, based on  Browder's demiclosedness principle, we prove for $r$-uniformly convex Banach spaces that the weak cluster points of the iterates 
		$x_{n+1}:=Tx_{n}$ belong to the fixed point set $\Fix T$ whenever the operator $T$ is nonexpansive and quasi $\alpha$-firmly. 
		If additionally the space has a Fr\'echet differentiable norm or satisfies Opial's property, then these iterates converge weakly to some element in $\Fix T$. 
		Further, the projections $P_{\Fix T}x_n$ converge strongly to this weak limit point. 
		Finally, we give three illustrative examples, where our theory can be applied, namely from
		infinite dimensional neural networks, semigroup theory, 
		and contractive projections in $L_p$, $p \in (1,\infty) \backslash \{2\}$ spaces on probability measure spaces.
	\end{abstract}

\section{Introduction}
Averaged operators  play an important role in the fixed point theory in Hilbert spaces. 
They emerged as a tool to obtain solutions to fixed point problems, 
where the underlying operator is not contractive 
and thus applying  Banach's fixed point theorem becomes inaccessible. 
The most notable appearance of averaged operators in the theory of Hilbert spaces is the classical method of iterations introduced by Krasnoselskij \cite{Krasnoselski} and Mann \cite{Mann}.
The notion of averaged operators can be naturally extended to any topological vector space, in particular to Banach spaces. 
However, unlike in the Hilbert setting, the usefulness of averaged operators in Banach spaces is not immediately obvious.
To make our point that such operators have many desirable properties 
related in particular to convergence theory, 
we introduce a related class of operators acting on $r$-uniformly convex Banach spaces, 
that we call $\alpha$-firmly nonexpansive. 
These operators extend an analogue notion for Hilbert spaces, see, e.g., \cite{LukTamTha18} for a pointwise variant, 
which in fact coincides with the concept of averaged operator there \cite[Chapter 5]{Bauschkebook}. 
For our general setting, we can show that an averaged operator is always $\alpha$-firmly nonexpansive. 
Notions of a firmly nonexpansive operators acting on Banach spaces date back to Bruck \cite{Bruck}, 
who introduced a definition that coincides with the usual one in a Hilbert space 
and was also used by Browder,  \cite[Definition 6]{Browder} under the name firmly contractive.
A somewhat stronger then Bruck's notion 
has found applications also in nonlinear metric spaces, 
in particular geodesic metric spaces, see Ariza-Ruiz  et al. \cite{Lopez,Lopez2} and Ba\v cak \cite{Bacak}, 
and  Reich and Shafrir  \cite{Reich-Shafrir} in hyperbolic spaces. 
In contrast to these definitions, our notation of $\alpha$-firmly nonexpansive and 
quasi $\alpha$-firmly nonexpansive operators  is less geometrically, but more analytically inspired. 
More precisely, it is dictated by the defining inequality for $r$-uniformly convex Banach spaces. 
In particular, Bruck's notion always implies ours, but not conversely. 
Our focus is on convergence properties of iterates generated by  quasi $\alpha$-firmly nonexpansive operators,
where we follow first the way of Opial's theorem \cite{Opial} to prove the desired convergence result. 
Then later by applying 
Browder's demiclosedness principle \cite{Browder-demi} we show that weak cluster points of a nonexpansive operator that has at least one fixed point and is asymptotic regular, all belong to the fixed point set of the operator. If additionally the space has a Fr\'echet differentiable norm then these iterates converge weakly to a certain element in the fixed point set.
We figured out some interesting relations of our approach to operators appearing in semigroup theory
and to contractive projections in $L_p$ spaces on probability measure spaces. 
The later one can be actually seen as conditional expectations \cite{Ando}.
In our future research, we will be also interested in $L_1$ and $L_\infty$ spaces.
At least for finite dimensional $L_\infty$ spaces, averaged operators were recently addressed
in connection with $L_\infty$ Laplacians in imaging in \cite{BHNS2019}. The notion of $\alpha$-firmly nonexpansive operator has also been considered in metric spaces, see for instance \cite{BerdellimaPhD}, \cite{Berdellima-paper} for the case of a CAT$(0)$ space and \cite{Luke} for a recent generalization to (non-linear) metric spaces.

This paper is organized as follows. 
In Section \ref{s:preliminaries} we collect preliminaries about $r$-uniformly convex Banach spaces. 
In particular, we address the relation between different notations and prove
an inequality for projections onto closed, convex sets.
Then, in Section \ref{s:alpha-firmly}, we introduce $\alpha$-firmly nonexpansive operators 
in $r$-uniformly convex Banach spaces. We show the relation to related definitions in the literature,
prove that compositions and convex combinations of such operators remain $\alpha$-firmly nonexpansive
and highlight the connection to  $\alpha$-averaged operators.
Section \ref{s:quasi} deals with quasi $\alpha$-firmly nonexpansive operators, where relations between fixed point sets
are central.
Fixed point iterations for  quasi $\alpha$-firmly nonexpansive operators are studied 
in Section \ref{s:fixedpoint}.
Starting with the proof that quasi $\alpha$-firmly nonexpansive operators are asymptotic regular,
we follows the path of Opial's convergence theorem. 
However, we provide another proof of the theorem
for uniformly convex Banach spaces by means of an auxiliary lemma. 
In general for $r$-uniformly convex Banach spaces that do not satisfy Opial's property we show that the weak cluster points of the iterates of a nonexpansive operator that is also quasi $\alpha$-firmly nonexpansive, all belong to the fixed point set of the operator. Moreover if the space has a Fr\'echet differentiable norm then these iterates converge weakly to a certain element in the fixed point set.
Finally, some possible directions for applying our theory are sketched in Section \ref{sec:illustrations}.

\section{$r$-Uniformly Convex Spaces}\label{s:preliminaries}
A normed space $(X,\|\cdot\|)$ is \emph{uniformly convex}, 
if for every $\varepsilon >0$, there exits $\delta(\varepsilon)>0$ 
such that for all $x,y \in X$  with $\|x\| = \|y\| = 1$ and 
$\|x-y\|\ge 2\varepsilon$ it holds 
$\frac12 \|x+y\| \le 1-\delta(\varepsilon)$. 
Uniformly convex spaces are reflexive and their dual spaces are uniformly convex, too.
Examples of uniformly convex spaces are $L_p$ and $\ell_p$, $p \in (1,\infty)$ as well as
Sobolev spaces $W_p^m$, $p \in (1,\infty)$ \cite{MQR2018}
and 
Orlicz spaces \cite{Kaminska}.
Therefore $L_1$, $L_\infty$ and $C[0,1]$ are not uniformly convex.
Every closed subspace of a uniformly convex Banach space is again uniformly convex.
Moreover, uniformly convex Banach spaces are have the Kadec--Klee property, also known as Radon--Riesz property, meaning that $f_n \weakly f$ and $\|f_n\| \rightarrow \|f\|$ implies $f_n \rightarrow f$.

The notation of uniformly convex normed spaces was introduced by Clarkson \cite{Clarkson}, who examined
integrals of functions mapping from the Euclidean space to a Banach space and established
an analogy to real-valued functions of bounded variation.
Uniformly convex spaces play an important role in approximation theory, since 
in such spaces there exists for every $x \in X$ an element of best approximation in a closed convex set
which is moreover unique, see, e.g. \cite{Cheney}. 
Indeed,  uniformly convex spaces are strictly convex, i.e. the relation
$\|\lambda x + (1-\lambda) y\| < 1$    
is fulfilled for all $x,y \in X$, $x \not = y$ with $\|x\| = \|y\| = 1$ and all $\lambda \in (0,1)$. 

In this paper, we are interested in special uniformly convex spaces.
To this end, we consider the so-called \emph{modulus of convexity} of $X$ given by
$$
\delta_X (\varepsilon) :=
\inf \left\{1- \tfrac12 \left\| x+y\right\|:  \|x\| = \|y\| = 1,\, 
\|x-y\|\ge 2\varepsilon\right\}.
$$
If there exists  $C >0$ and $r \in [2,\infty)$ 
such that 
$\delta_X (\varepsilon) \ge   \left( \frac{\varepsilon}{C} \right)^r$ for every $\varepsilon \in (0,1]$,
then $(X,\|\cdot\|)$ has a \emph{modulus of convexity of power type} $r$. 
By a result of Pisier \cite{Pisier} each uniformly convex Banach space 
admits an equivalent norm with a modulus of convexity of power type $r$ for some $r \ge 2$.
There exist several equivalent definitions in the literature.
For $r \ge 2$, Borwein et al. \cite{Borwein} showed that a Banach space $(X,\|\cdot\|)$ has a modulus of convexity of power type $r$ if and only if the function $\| \cdot\|^r$ is uniformly convex, 
meaning that for every $\varepsilon>0$ it holds 
\begin{equation}	\label{eq:r-unif_borwein}
	\inf_{x,y} \left\{ \frac{\|x\|^r+\|y\|^r}{2} - \Big\|\frac{x+y}{2}\Big\|^r :
	\|x-y\| = 2 \varepsilon \right\}>0.
\end{equation}
It was proved by Ball et al. \cite{Ball} that $(X,\|\cdot\|)$ has a modulus of convexity of power type $r$ if and only if there is a constant $K>0$ such that
\begin{equation}	\label{eq:r-unif}
	\Big\|\frac{x+y}{2}\Big\|^r+\Big\|\frac{x-y}{2K}\Big\|^r \le \frac{\|x\|^r+\|y\|^r}{2}
\end{equation}
for all $x,y\in X$.
Clearly, by the parallelogram law in Hilbert spaces 
\begin{equation}	\label{eq:paral}
	\Big\|\frac{x+y}{2}\Big\|^2+\Big\|\frac{x-y}{2}\Big\|^2 = \frac{\|x\|^2+\|y\|^2}{2},
\end{equation}
these spaces have modulus of convexity of power type $r=2$.
By Clarkson's inequalities \cite{Clarkson},
the spaces $L_p$ and $\ell_p$ have modulus of convexity of power type $r=p$ for $p \in [2,\infty)$ 
and modulus of convexity of power type $r=q=p/(p-1)$ for $p \in (1,2)$,
where $\frac1p + \frac1q = 1$. 
In both cases, we can choose $K=1$ in \eqref{eq:r-unif}.
Actually, it was shown that for $p \in (1,2]$, the $L_p$ spaces are 2-uniformly convex for $p \in (1,2]$ and we have $K= 1/\sqrt{p-1}$, see \cite{Ball,hanner1956}. Other examples of uniformly convex spaces are the Sobolev spaces $W_p^m$, $p \in (1,\infty)$.

We will use another definition.
For $r\in[2,+\infty)$ we say that $(X,\|\cdot\|)$ is an \emph{$r$-uniformly convex space}, 
if there exists a constant $c_r>0$ such that for all $w\in[0,1]$ and all $x,y\in X$ it holds
\begin{equation} \label{eq:p1}
\|(1-w)x + w y\|^r\leq (1-w)\|x\|^r + w\|y\|^r-\frac{c_r}{2}w(1-w)\|x-y\|^r.
\end{equation}
Setting $x=0$ and $y \not = 0$, it follows immediately from the definition that $c_r \le 2$.
Definition \eqref{eq:p1} was used in the more general setting of
geodesic spaces, where the norm is replaced by the geodesic distance, 
by Naor and Siberman \cite{naor_silberman_2011} and also by Ariza-Ruiz et al. \cite{Lopez, Lopez2}.
Note that in a Hilbert space $(X,\| \cdot \|)$, we have equality in \eqref{eq:p1} for $r=2$,
$c_2 = 2$ and all $w \in [0,1]$.
Indeed, by the next proposition, $(X,\| \cdot \|)$  has modulus of convexity of power type $r$
if and only if the space is an $r$-uniformly convex.

\begin{proposition} 	\label{p:p-uniformcvx}
Equation \eqref{eq:p1} with $w = \frac12$ implies \eqref{eq:r-unif} with $K$ determined by $c_r =  \frac{8}{(2K)^r}$.
If	\eqref{eq:r-unif} is fulfilled, then \eqref{eq:p1} holds true with $c_r =  \frac{4}{(2K)^r}$
and with $c_2 =  \frac{8}{(2K)^2}$ if $r=2$.
\end{proposition}

\begin{proof}
The first implication is straightforward.

For the second part, we follow the lines of \cite[Remark 2.1]{Zalinescu}.
Suppose that \eqref{eq:r-unif} is fulfilled. 
Then we have for $w\in[0,1/2]$ by convexity of $\| \cdot\|^r$ 
and \eqref{eq:r-unif} that
\begin{align*}
\|(1-w)x+wy\|^r 
&= \Big\|(1-2w)x+2w\Big(\frac{x+y}{2}\Big)\Big\|^r\\
&\leq (1-2w)\|x\|^r+2w\Big\|\frac{x+y}{2}\Big\|^r\\
&\leq (1-2w)\|x\|^r+2w\Big(\frac{1}{2}\|x\|^r+\frac{1}{2}\|y\|^r-\Big \| \frac{x-y}{2K} \Big \|^r)\Big)\\
&=(1-w)\|x\|^p+w\|y\|^r-2w \Big \| \frac{x-y}{2K} \Big \|^r\\
&\le (1-w)\|x\|^r+w\|y\|^r-\frac{c_r}{2}w(1-w)\|x-y\|^r
\end{align*}
for $c_r \le 4/\left((2K)^r(1-w) \right)$. 
The largest constant $c_r$ such that this inequality is fulfilled for all $w\in[0,1/2]$ is
$c_r =  \frac{4}{(2K)^r}$. Note that we get $c_r =  \frac{8}{(2K)^r}$ for $w = \frac12$.

Similarly we can argue for $w\in[1/2,1]$. 

For $r=2$, it can be shown by standard arguments on midpoint convexity
that \eqref{eq:p1} holds true for all $w \in [0,1]$
if and only if it holds true for $w=\frac12$. Then the final claim follows from the above
inequality.
\end{proof}

As already mentioned, projections onto closed, convex sets in uniformly convex spaces are uniquely determined. 
In $r$-uniformly convex spaces they fulfill the additional relation stated in the following proposition.

\begin{proposition} 	\label{p:projection}
	Let $C\subseteq X$ be a closed, convex set in an $r$-uniformly convex space $(X,\| \cdot\|)$. 
	Then, for every $x \in X$,  the orthogonal projection $P_C x$ of $x$ onto $C$ is nonempty and consists of one element. 
	Moreover, it holds for all $x\in X$ and all $y\in C$ that
	\begin{equation}	\label{eq:identity}
	\|x-P_Cx\|^r+\frac{c_r}{2}\|P_Cx-y\|^r\leq\|x-y\|^r.
	\end{equation}
	and for all $x,y \in \in X$ that
	\begin{equation}	\label{eq:identityprojection}
	\|P_Cx-P_Cy\|^r\leq\frac{1}{c_r}\Big(\|x-P_Cy\|^r+\|y-P_Cx\|^r-\|x-P_Cx\|^r-\|y-P_Cy\|^r\Big).
	\end{equation}
\end{proposition}

\begin{proof}
	It remains to prove \eqref{eq:identity}. Then relation \eqref{eq:identityprojection} follows immediately.
	For $x \in C$,  the assumption follows since $c_r \le 2$.
	For  $x\in X\setminus C$ and $y\in C\setminus P_Cx$,
	let $(1-w)P_Cx+wy$, $w\in (0,1)$. 
	Then $(1-w)P_Cx+wy\in C$ since $C$ is convex. By $p$-strongly convexity of $\|\cdot\|^r$, we get
	$$
	\|x-((1-w)P_Cx+wy)\|^r
	\leq
	(1-w)\|x-P_Cx\|^r + w\|x-y\|^r-\frac{c_r}{2}(1-w)w\|P_Cx-y\|^r.$$
	Together with 
	$\|x-P_Cx\|\leq \|x-((1-w) P_Cx+wy)\|$ 
	this implies 
	$$w\|x-P_Cx\|^r+\frac{c_r}{2}(1-w)w\|P_Cx-y\|^r 
	\leq w\|x-y\|^r.$$ 
	Dividing by $w$ and taking limit as $w$ goes to zero, we  get the assertion \eqref{eq:identity}.
\end{proof}

\begin{remark}[Projections in Banach spaces]\label{rem:hilbert}
	It is known that projections are nonexpansive, whenever $X$ is a Hilbert space. 
	In the converse, it is known that if projections onto closed, convex sets in a Banach space $X$ of dimension $\geq 3$ are nonexpansive, then $X$ must be a Hilbert space \cite[Theorem 5.2]{Phelps}. 
In fact we have even the stronger statement that in a non-Hilbert Banach space $X$ of dimension $\geq 3$ no bounded, smooth, closed and convex subset $C$ of $X$ with a nonempty interior is a nonexpansive retract of $X$, that is there exists no nonexpansive mapping $R:X\to C$ such that $Rx=x$ for all $x\in C$  \cite[Theorem 1]{Bruck-retract}.  
For the case of dimension $2$, we can refer to \cite{Gruber} for an exhaustive analysis of nonexpansive retracts. In particular, only the so-called radial projections on certain sets
are nonexpansive mappings under special conditions \cite[Theorem 2]{Gruber}. For a study of various aspects of nonexpansive retracts and retractions in certain Banach and metric spaces, with special emphasis on the so called compact nonexpansive envelope property we refer to \cite{Reich}.

While in general projections in Banach spaces fail to be nonexpansive, in uniformly convex Banach spaces, hence also in $r$-uniformly convex spaces, it can be shown that projections are uniformly continuous on bounded sets, see \cite[Theorem 3.10]{Bacak-Kohlenbach}.
	\end{remark}

\section{$\alpha$-Firmly Nonexpansive Operators} \label{s:alpha-firmly}
Let $(X,\|\cdot\|)$ be an $r$-uniformly convex space and $D,E\subseteq X$ nonempty sets. 
For $\alpha \in (0,1)$, an operator $T:X\to X$ is said to be \emph{$\alpha$-firmly nonexpansive
on} $D\times E$ if for
all $x \in D$ and all $y\in E$ it holds
 \begin{equation} 	\label{eq:firm1}
 	\|Tx-Ty\|^r \leq \|x-y\|^r - \frac{c_r}{2}\frac{1-\alpha}{\alpha}\|(\Id-T)x-(\Id-T)y\|^r.
 \end{equation}
If $D = E = X$, we simply say that $T$ is \emph{$\alpha$-firmly nonexpansive}.
If $T:X\to X$ is $\alpha$-firmly nonexpansive for some $\alpha\in(0,1)$, 
then we see, since $\frac{1-\alpha}{\alpha}$, $\alpha \in (0,1)$ is monotone decreasing, 
that $T$ is $\alpha$-firmly nonexpansive for any $\alpha'\in(\alpha,1)$.

\begin{remark}\label{rem:bruck}
There exist several definitions of firmly nonexpansive operators on Banach spaces in the literature.
Let 
\begin{equation} \label{eq:Bruck1}
\varphi(w;T,x,y):=\|(1-w)x+wTx-((1-w)y+wTy)\|.
\end{equation}
Note that $\varphi(0;T,x,y) = \|x-y\|$ and $\varphi(1;T,x,y) = \|Tx-Ty\|$.
For a fixed $w \in [0,1)$, let us say that an operator $T:X\to X$ is \emph{$w$-firmly nonexpansive}, if 
\begin{equation}\label{eq:Bruck}
\varphi(1;T,x,y) \leq \varphi(w;x,y)
\end{equation}
for all $x,y \in X$. 
Then Bruck \cite{Bruck} called the operator $T$ \emph{firmly nonexpansive} 
if \eqref{eq:Bruck} holds true for all $w \in [0,1)$.
This concept of firmly nonexpansiveness was also used by Browder \cite{Browder} under the name firmly noncontractive. 

A stronger condition was required in nonlinear spaces, namely by
Reich and Shafier \cite{Reich-Shafrir} in hyperbolic spaces and by Ba\'cak \cite{Bacak} and Ariza-Ruiz et al. \cite{Lopez}
in Hadamard spaces. 
These authors called $T:X\to X$ firmly nonexpansive,
if the function $\varphi(\cdot;T,x,y)$ is nonincreasing on $[0,1]$ for all $x,y \in X$.
Indeed, for Banach spaces these definitions coincide.
\end{remark}

By the following proposition, the relation \eqref{eq:Bruck} implies that $T$ is $\alpha$-averaged
for some $\alpha \in [\frac12,1)$. 

 \begin{proposition}	\label{p:Bruck}
	Let $(X,\|\cdot\|)$ be a $r$-uniformly convex space. 
	If the operator $T:X \rightarrow X$ fulfills \eqref{eq:Bruck} for some $w \in [0,1)$,
	then it is $\alpha$-firmly nonexpansive for $\alpha \in [\frac{1}{1+w},1)$. 
	In particular, if $T:X \rightarrow X$ fulfills \eqref{eq:Bruck} for all $w \in [0,1)$,
	i.e., if $T$ is firmly nonexpansive in the sense of Bruck, then it is $\alpha$-firmly nonexpansive for all
	$\alpha \in [\frac12,1)$.
\end{proposition}

\begin{proof}
Let $T:X\to X$ fulfill \eqref{eq:Bruck}, i.e., 
$
\|Tx-Ty\|
\leq 
\varphi(w;x,y)
$
for all $x,y \in X$ and some $w\in[0,1)$.	
Since $(X,\|\cdot\|)$ is $p$-uniformly convex, where $p\in[2,+\infty]$, we have on the other hand
\begin{align*}
\varphi(w;x,y)^r
&=\|(1-w)(x-y)+w(Tx-Ty)\|^r\\
&\leq  (1-w)\|x-y\|^r+w\|Tx-Ty\|^r-\frac{c_r}{2}w(1-w)\|(\Id-T)x-(\Id-T)y\|^r.	
\end{align*}
Combining both inequalities, we obtain 
\begin{align}
(1-w)\|Tx-Ty\|^r &\leq (1-w)\|x-y\|^r-\frac{c_r}{2}w(1-w)\|(\Id-T)x-(\Id-T)y\|^r,\\
\|Tx-Ty\|^r 
&\leq \|x-y\|^r-\frac{c_r}{2}w\|(\Id-T)x-(\Id-T)y\|^r \\
&\le \|x-y\|^r-\frac{c_r}{2} \frac{1-\alpha}{\alpha} \|(\Id-T)x-(\Id-T)y\|^r 
\end{align}
for $w \ge \frac{1-\alpha}{\alpha}$, resp. $\alpha \ge \frac{1}{1+w}$.
\end{proof}

By the next proposition, convex combinations and compositions of $\alpha$-firmly nonexpansive operators
are again $\alpha$-firmly nonexpansive.

\begin{proposition}	\label{p:convexcombinations}
	Let $T_1,T_2,...,T_n:X\to X$ be $\alpha$-firmly nonexpansive with constants $\alpha_i\in(0,1)$
	and $w_i \in [0,1]$, 	$i=1,2,...,n$ such that $\sum_{i=1}^n w_i = 1$. 
	Then the following relations hold true:
	\begin{itemize}
	\item[i)]
	The operator $T:=\sum_{i=1}^n w_i T_i$ is $\alpha$-firmly nonexpansive with constant 
	$\alpha = \alpha_{\mathrm{max}} :=\max\{\alpha_i:i=1,\ldots,n\}$.
	\item[ii)]
	The operator $T:=T_n...T_2T_1$ is $\alpha$-firmly nonexpansive with constant 
	$\alpha := \Big(1+\displaystyle\frac{1-\alpha_{\mathrm{max}}}{n^{r-1}\alpha_{\mathrm{max}}}\Big)^{-1}$, . 
	\end{itemize}
The same relations are fulfilled for quasi $\alpha$-firmly nonexpansive operators $T_1,T_2,...,T_n:X\to X$.
\end{proposition}

\begin{proof}
i) 	By convexity of $\| \cdot \|^r$, $p\ge 1$, we have
	\begin{align*}
		\|Tx-Ty\|^r& \leq \sum_{i=1}^n w_i\|T_ix-T_iy\|^r
	\end{align*}
	for all $x,y \in X$ and  since the $T_i$ are $\alpha$-firmly nonexpansive, we get
			\begin{align*}
		\sum_{i=1}^nw_i\|T_ix-T_iy\|^r\leq \|x-y\|^r-\frac{c_r}{2}
		\sum_{i=1}^n w_i\frac{1-\alpha_i}{\alpha_i}\|(\Id-T_i)x-(\Id-T_i)y\|^r,
	\end{align*}
for all $x,y\in X$.
We have 
\begin{equation} \label{helper}
\frac{1-\alpha_{\max}}{\alpha_{\max}} \leq \frac{1-\alpha_i}{\alpha_i}
\end{equation} 
for all $i=1,2,...,n$ 
and consequently
	\begin{align*}
	\sum_{i=1}^n w_i\|T_i x-T_i y\|^r\leq \|x-y\|^r-\frac{c_r}{2} \frac{1-\alpha}{\alpha}\sum_{k=1}^n w_i\|(\Id-T_i)x-(\Id-T_i)y\|^r.
	\end{align*}
so that we get by convexity of $\| \cdot \|^r$, $p\ge 1$ for all $x,y\in X$ finally
		\begin{align*}
	\sum_{k=1}^n w_i\|T_i x-T_i y\|^r\leq \|x-y\|^r-\frac{c_r}{2}\cdot \frac{1-\alpha}{\alpha}\|(\Id-T)x-(\Id-T)y\|^r.
	\end{align*}
	ii)
	Let $S_k:=T_k...T_2T_1$ for $k=1,2,...,n$ with the convention $S_0:=\Id$.
	Since the  $T_i$ are $\alpha$-firmly and by \eqref{helper}, we obtain
	\begin{align*}
		\|Tx-Ty\|^r&=\|T_n...T_2T_1x-T_n...T_2T_1y\|^r\\
		&\leq \|S_{n-1}x-S_{n-1}y\|^r-\frac{c_r}{2} \frac{1-\alpha_n}{\alpha_n}\|(\Id-T_n)S_{n-1}x-(\Id-T_n)S_{n-1}y\|^r \\
		&....\\
		&\leq 
		\|x-y\|^r-\frac{c_r}{2}\sum_{i=1}^n\frac{1-\alpha_i}{\alpha_i}\|(\Id-T_i)S_{i-1}x-(\Id-T_i)S_{i-1}y\|^r\\
	 &\leq \|x-y\|^r-\frac{c_r}{2}\frac{1-\alpha_{\max}}{\alpha_{\max}}\sum_{i=1}^n\|(\Id-T_i)S_{i-1}x-(\Id-T_i)S_{i-1}y\|^r.
	\end{align*}
By convexity of $\| \cdot\|^r$, $p\ge1$, we get 
\begin{align*}
\|(\Id-T)x-(\Id-T)y\|^r
&=
n^r\Big\|\frac{1}{n}\sum_{i=1}^n(\Id-T_i)S_{i-1}x-(\Id-T_i)S_{i-1}y\Big\|^r\\
&\leq n^{r-1}\sum_{i=1}^n\|(\Id-T_i)S_{i-1}x-(\Id-T_i)S_{i-1}y\|^r.
\end{align*}
Noting that 
$$(\Id-T)=\sum_{i=1}^n(\Id-T_i)S_{i-1},$$	
and setting $\alpha:=\Big(1+\displaystyle\frac{1-\alpha_{\max}}{ n^{r-1} \alpha_{\max} }\Big)^{-1}\in(0,1)$,
this yields 
\begin{align*}
	\|Tx-Ty\|^r
	&\leq\|x-y\|^r-\frac{c_r}{2n^{r-1}} \frac{1-\alpha_{\max}}{\alpha_{\max}}\|(\Id-T)x-(\Id-T)y\|^r \\
&\leq \|x-y\|^r-\frac{c_r}{2}\frac{1-\alpha}{\alpha}\|(\Id-T)x-(\Id-T)y\|^r \\
\end{align*}  
for all $x,y\in X$.
\end{proof}

The concept of $\alpha$-firmly nonexpansive operators is closely related to those of 
$\alpha$-averaged operators.
Recall that an operator $T:X\to X$ is \emph{$\alpha$-averaged with averaging constant} $\alpha\in(0,1)$,
 if there exists a nonexpansive operator $R:X\to X$ such that 
	\begin{equation} 	\label{eq:avg}
	T=(1-\alpha)\Id+\alpha R.
	\end{equation}
	
\begin{remark}	
For Hilbert spaces $X= H$  it can be shown that \eqref{eq:avg} holds true if and only if $T:H \rightarrow H$ fulfills
\begin{equation} 	\label{eq:firm_hilbert}
 	\|Tx-Ty\|^2 \leq \|x-y\|^2 - \frac{1-\alpha}{\alpha}\|(\Id-T)x-(\Id-T)y\|^2,
 \end{equation}
see, e.g. \cite[Proposition 4.25]{Bauschkebook}. Therefore, with $r=2$ and $c_2=2$, 
the $\alpha$-firmly nonexpansive operators coincide with the $\alpha$-averaged operators on Hilbert spaces.
\end{remark}

We will see, that the remark  is no longer true for general $r$-uniformly convex spaces.
But first let us consider some useful properties of averaged operators.
These properties were proved with different averaging constants for Hilbert spaces in 
\cite{Bauschkebook,Comb-Yamada} based on the equivalent relation \eqref{eq:firm_hilbert}.
Since we have not found a proof for general Banach spaces in the literature, we give it in the appendix.

\begin{proposition}	\label{p:avgcvx}
	Let $T_1,...,T_n:X\to X$ be $\alpha$-averaged operators with averaging constants $\alpha_i\in(0,1)$ 
	and $w_i \in [0,1]$, 	$i=1,2,...,n$ such that $\sum_{i=1}^n w_i = 1$. 
	Then the following relations hold true:
	\begin{itemize}
	\item[i)]
	The operator $T:=\sum_{i=1}^n w_i T_i$ is $\alpha$-averaged with averaging constant $\alpha:=\sum_{i=1}^n w_i\alpha_i$.
	\item[ii)]
	Then the operator $T:=T_n...T_2T_1$ is $\alpha$-averaged with averaging constant $\alpha = 1- \Pi_{i=1}^n (1- \alpha_i)$. 
	\end{itemize}
\end{proposition}

The next proposition states that the set of $\alpha$-firmly nonexpansive operators include those of $\alpha$-averaged operators. 

\begin{proposition}\label{prop:relation}
	\label{p: 1}
	Let $(X,\|\cdot\|)$ be a $r$-uniformly convex space. Then 
	an $\alpha$-averaged operator is $\alpha$-firmly nonexpansive.
\end{proposition}

\begin{proof}
	Let $T:X\to X$ be an $\alpha$-averaged operator with averaging constant $\alpha\in(0,1)$. 
	By \eqref{eq:avg} there exists a nonexpansive operator $R:X\to X$ such that $T:=(1-\alpha)\Id+\alpha R$.
	For $x,y\in X$, we obtain by \eqref{eq:p1} with $w = \alpha$ that
	\begin{align*}
		 \|Tx-Ty\|^r\leq (1-\alpha)\|x-y\|^r+\alpha\|Rx-Ry\|^r-\frac{c_r}{2}\alpha(1-\alpha)\|(\Id-R)x-&(\Id-R)y\|^r
	\end{align*}
	and rearranging terms 
	\begin{align*}
\alpha(\|x-y\|^r-\|Rx-Ry\|^r)\leq \|x-y\|^r-\|Tx-Ty\|^r-\frac{c_r}{2}\frac{1-\alpha}{\alpha^{r-1}}\|(\Id-T)x-&(\Id-T)y\|^r.
	\end{align*}
	Now the nonexpansivity of $R$ and the fact that $\alpha^{r-1}\leq\alpha$ for $r \geq 2$ and $\alpha\in[0,1]$ yields
the assertion
	\begin{align*}
0\leq \|x-y\|^r-\|Tx-Ty\|^r-\frac{c_r}{2} \frac{1-\alpha}{\alpha}\|(\Id-T)x-&(\Id-T)y\|^r.
	\end{align*}
\end{proof}

The following examples show that the reverse inclusion is in general not true.

\begin{example}
Equip $\mathbb{R}^n$ with the norm $\|x\|:=(|x_1|^r+|x_2|^r+...+|x_n|^r)^{1/r}$ 
for some $p\geq2.$ Let $\mathbb{B}(0,r)$ 
be the closed ball at $0$ of radius $r<1$. Let $R:\mathbb{R}^n\to \mathbb{R}^n$ be the operator defined as 
\begin{align}
Rx:=\left\{\begin{array}{ll}
-\displaystyle\frac{x}{r}\;& x\in \mathbb{B}(0,r)\\
0\;&x\notin \mathbb{B}(0,r)
\end{array}\right.
\end{align}
Given $\alpha\in(0,1)$ define 
\begin{equation}
	\label{l:T}
	Tx:=(1-\alpha)x+\alpha Rx\quad\forall x\in\mathbb{R}^n.
\end{equation}
Clearly the operator $R$ is not nonexpansive. Define the set $D(r):=\{x\in \mathbb{R}^n\;|\;\|y\|>1+r\}$. 
Note that for all $x\in \mathbb{B}(0,r), y\in D(r)$ we have
\begin{align*}
\|Tx-Ty\|^r&=\|(1-\alpha-\alpha \frac{1}{r})x-(1-\alpha)y\|^r=\|-\alpha \frac{x}{r}+(1-\alpha)(x-y)\|^r\\&
\leq \alpha \frac{1}{r^r}\|x\|^r+(1-\alpha)\|x-y\|^r-\frac{c_r}{2}\alpha(1-\alpha)\|(1+\frac{1}{r})x-y\|^r\\
&\leq \alpha +(1-\alpha)\|x-y\|^r-\frac{c_r}{2}\alpha(1-\alpha)\|(1+\frac{1}{r})x-y\|^r\\
&\leq\alpha(1-\|x-y\|^r)+\|x-y\|^r-\frac{c_r}{2}\frac{1-\alpha}{\alpha}\|(\Id-T)x-(\Id-T)y\|^r\\
&\leq\|x-y\|^r-\frac{c_r}{2}\frac{1-\alpha}{\alpha}\|(\Id-T)x-(\Id-T)y\|^r. 
\end{align*}
Therefore $T$ is $\alpha$-firmly nonexpansive on $\mathbb{B}(0,r)\times D(r)$.
\end{example}

\begin{example}
	Let $(X,\|\cdot\|)$ be an $\ell_p$ space for some $p>2$. 
	Let $T:X\to X$ be an operator acting by the formula $Tx:=(x_1,0,...,0,...)$ for all $x\in X$ where $x:=(x_1, x_2,...)$. 
	Let $\alpha_r:=c_r/(c_r+2)$ then $T$ is $\alpha$-firmly nonexpansive with $\alpha:=\alpha_r$. Indeed 
	\begin{align*}
	\|Tx-Ty\|^r+\frac{c_r}{2}\frac{1-\alpha_r}{\alpha_r}\|(\Id-T)x-(\Id-T)y\|^r&=\|Tx-Ty\|^r+\|(\Id-T)x-(\Id-T)y\|^r\\
	&
	=|x_1-y_1|^r+\sum_{i=2}^{\infty}|x_i-y_i|^r\\
	&
	=\sum_{i=1}^{\infty}|x_i-y_i|^r =\|x-y\|^r.
	\end{align*}
Suppose that $R:X\to X$ is some operator such that $T:=(1-\alpha_r)\Id+\alpha_r R$. By $r$-uniform convexity of $(X,\|\cdot\|)$ we obtain 
\begin{align*}
\|Tx-Ty\|^r\leq (1-\alpha_p)\|x-y\|^r+\alpha_p\|Rx-Ry\|^r-\frac{c_r}{2}\alpha_p(1-\alpha_p)\|(\Id-R)x-
&(\Id-R)y\|^r.
\end{align*}
Rearranging terms and using identity \eqref{eq:avg} for the operator $R$ we get
\begin{align*}
\alpha_r(\|x-y\|^r-\|Rx-Ry\|^r)
\leq 
\|x-y\|^r-\|Tx-Ty\|^r-\frac{c_r}{2}\frac{1-\alpha_p}{\alpha_p^{p-1}}\|(\Id-T)x-
&(\Id-T)y\|^r.
\end{align*}
By the fact that $\alpha_p^{r-1}<\alpha_p$ for $r>2$ we obtain for all $x,y\neq 0$ 
that
$$\alpha_r(\|x-y\|^r-\|Rx-Ry\|^r)<\|x-y\|^r-\|Tx-Ty\|^r-\frac{c_r}{2} 
\frac{1-\alpha_r}{\alpha_r}\|(\Id-T)x-(\Id-T)y\|^r.$$	
By the above calculations the right hand side vanishes. 
Therefore $\|x-y\|^r<\|Rx-Ry\|^r$ for all $x,y\neq 0$. 
Hence $R$ cannot be a nonexpansive operator.
\end{example}

\section{Quasi $\alpha$-Firmly Nonexpansive Operators}\label{s:quasi}
%
Let $(X,\|\cdot\|)$ be an $r$-uniformly convex space.
For $\alpha \in (0,1)$, an operator $T:X\to X$ is said to be
\emph{quasi $\alpha$-firmly nonexpansive}, if $\Fix T\neq\emptyset$ and
for all $x\in X$ and all $y\in\Fix T$ it holds
 \begin{equation} 	\label{eq:firm1_quasi}
 	\|Tx-Ty\|^r \leq \|x-y\|^r - \frac{c_r}{2}\frac{1-\alpha}{\alpha}\|(\Id-T)x-(\Id-T)y\|^r.
 \end{equation}
It follows directly from the definition that
a quasi $\alpha$-firmly nonexpansive operator $T:X\to X$ fulfills for all $x\in X$ and all $y\in\Fix T$ the relation
	\begin{equation}	\label{eq:quasi}
	\|Tx-y\|^r \leq \|x-y\|^r-\frac{c_r}{2}\frac{1-\alpha}{\alpha}\|Tx-x\|^r.
	\end{equation}
We say that $T:X\to X$ is \emph{quasi nonexpansive}, if $\Fix T\neq\emptyset$ and for all $x\in X$ and all $y\in\Fix T$,
$$
\|Tx-y\| \leq \|x-y\|.
$$	

\begin{lemma}	\label{l:fixedpointset}
	Let $T:X\to X$ be a quasi nonexpansive operator. Then $\Fix T$ is a nonempty closed, convex set. 
\end{lemma}

\begin{proof}	
	By definition $\Fix T$ is nonempty.
	Let $(x_n)_{n \in \mathbb N} \subseteq\Fix T$ 
	such that $x_n\to x$ for some $x\in X$.
	Then, by the triangle inequality and since $T$ is quasi nonexpansive, we obtain
	$$0\leq\|Tx-x\|\leq \|Tx-x_n\|+\|x_n-x\|\leq 2\|x_n-x\|\to 0\;\text{as}\;n\rightarrow +\infty.$$
	Therefore $\Fix T$ is a closed set. 
	Now let $w\in[0,1]$ and $x_1,x_2\in\Fix T$. 
	Consider the element $x_w:=(1-w)x_1+wx_2\in X$. 
	By the triangle inequality and the quasi nonexpansiveness of $T$, we conclude
	$$\|x_1-x_2\|\leq \|x_1-Tx_w\|+\|Tx_w-x_2\|\leq \|x_1-x_w\|+\|x_w-x_2\|=\|x_1-x_2\|.$$
	Therefore, $Tx_w$ is on the line segment $[x_1,x_2]$ joining $x_1$ with $x_2$. Hence we get
	\begin{align*}
& \|x_2-x_w\|+\|x_w-Tx_w\|= \|x_2-Tx_w\|\leq \|x_2-x_w\|\quad\text{if}\quad Tx_w\in[x_1,x_w],\\
&\|Tx_w-x_w\|+\|x_w-x_1\|= \|Tx_w-x_1\|\leq \|x_w-x_1\|\quad\text{if}\quad Tx_w\in[x_w,x_2],
\end{align*}
and  consequently $Tx_w=x_w$ so that $\Fix T$ is convex.
\end{proof}

The next proposition describes the fixed point sets of the composition and convex combination 
of finitely many quasi $\alpha$-firmly nonexpansive operators.

\begin{proposition}	\label{p:fixedpointset-compositions}
Let $T_i:X\to X$ be quasi $\alpha$-firmly nonexpansive operators 
with constants $\alpha_i\in(0,1)$, $i=1,2,...,n$ and $w_i \in [0,1]$, 	$i=1,2,...,n$ such that $\sum_{i=1}^n w_i = 1$.
Let $T\coloneqq T_n T_{n-1}...T_2T_1$ and $S := \sum_{i=1}^n w_i T_i$.
If $\bigcap_{i=1}^n\Fix T_i\neq\emptyset$, 
then $\Fix T = \Fix S = \bigcap_{i=1}^n\Fix T_i$. 
In particular, $\Fix T$ is a nonempty, closed, convex set. 
\end{proposition}

\begin{proof}
	It is clear that $\bigcap_{i=1}^n\Fix T_i\subseteq \Fix T$ and $\bigcap_{i=1}^n\Fix T_i\subseteq \Fix S$. 
	By assumption, this implies that $\Fix T$  and $\Fix S$ are nonempty.
	\\
	1.
	We show the other direction for $\Fix T$ by induction on $n\in\mathbb{N}$. 
	For $n=1$ the claim is evident. 
	Suppose that the statement is true for the composition of  $k \le n-1$ operators. 
	Set $\tilde T \coloneqq T_{n-1}T_{n-2}...T_2 T_1$. 
	Take $x\in \Fix T$. 
	Then we have three mutually exclusive situations:
	First, $\tilde T x \in \Fix T_n$, 
	which implies $\tilde T x = T_n \tilde T x = Tx=x$. Therefore $x\in \Fix \tilde T \cap \Fix T_n$ 
	and by induction hypothesis $x\in \bigcap_{i=1}^n\Fix T_i$. 
	Second, let $x \in \Fix \tilde T$. Then 
	$x = T x = T_n \tilde T x =T_n x$ 
	implies $x \in\Fix T_n$ and consequently $x\in\Fix T_n \cap \tilde T=\bigcap_{i=1}^n\Fix T_i$. 
	Third, consider $x\notin \Fix \tilde T$ and $\tilde T x \notin \Fix T_n$.
	Then, for any $y \in \Fix T_n\cap \Fix \tilde T$, we have by \eqref{eq:quasi} that
		\begin{align*}
	\|x-y\|^r=\|T_n \tilde T x - T_n y\|^r
	&\leq 
	\|\tilde T x - y\|^r-\frac{c_r}{2}\,\frac{1-\alpha_n}{\alpha_n} \|\tilde T x - T_n \tilde T x\|^r
		< \|\tilde T x - \tilde T y\|^r
	\end{align*}	 
	Since $y \in \Fix \tilde T = \bigcap_{i=1}^{n-1}\Fix T_i$ and $\tilde T$ is the composition of quasi nonexpansive operators,
	it is easy to check that $\tilde T$ is also quasi nonexpansive.
	Hence we get the contradiction
	$
	\|x-y\|^r < \|x-y\|^r
	$.
In summary, this yields	$\Fix T = \bigcap_{i=1}^n\Fix T_i$.
\\
2. Now let $x\in\Fix S$ and $y\in\bigcap_{i=1}^n\Fix T_i$. Then 
	\begin{align*}
	\|x-y\|^r&=\|Tx-Ty\|^r=\|\sum_{i=1}^n w_i(T_ix-T_iy)\|^r
	\leq\sum_{i=1}^nw_i\|T_ix-T_iy\|^r\\
	&\leq \|x-y\|^r-\frac{c_r}{2}\sum_{i=1}^n w_i\frac{1-\alpha_i}{\alpha_i}\|x-T_i x\|^r.
	\end{align*}
From the last inequality we obtain 
	$$\frac{c_r}{2}\sum_{i=1}^n w_i\frac{1-\alpha_i}{\alpha_i}\|x-T_ix\|^r\leq 0,$$
	which is true if and only if $T_i x=x$ for every $i=1,2,...,n$. 
	Thus $x\in\bigcap_{i=1}^n\Fix T_i$. 
	\end{proof}

As an immediate implication of Proposition \ref{p:fixedpointset-compositions} we obtain the following 
calculus rules for the class of quasi $\alpha$-firmly nonexpansive operators.

\begin{proposition}	\label{p:calculus-quasi}
	Let $T_i:X\to X$ be quasi $\alpha$-firmly nonexpansive operators 
with constants $\alpha_i\in(0,1)$, $i=1,2,...,n$ and $w_i \in [0,1]$, 	$i=1,2,...,n$ such that $\sum_{i=1}^n w_i = 1$.
Let $T\coloneqq T_n T_{n-1}...T_2T_1$ and $S := \sum_{i=1}^n w_i T_i$.
If $\bigcap_{i=1}^n\Fix T_i\neq\emptyset$, 
then $s$ and $T$ are also quasi $\alpha$-firmly nonexpansive operators.
\end{proposition}	

\begin{proof}
	Note that by Proposition \ref{p:fixedpointset-compositions}, we have $\bigcap_{i=1}^n\Fix T_i=\Fix T= \Fix S$.
	The rest is analogous as the proof of Proposition \ref{p:convexcombinations}. 
\end{proof}

\section{Fixed Point Theorems} \label{s:fixedpoint}
%
In this section, we are interested in fixed point iterations of quasi $\alpha$-firmly nonexpansive operators.
We start with the important observations that these operators are asymptotic regular,
a property which is essential for the convergence when following ideas of Opial's convergence theorem.
We will give a different proof of Opial's well-known theorem, and
address in a corollary, the convergence of the iterates produced by our quasi $\alpha$-firmly nonexpansive operators.
In the second part of this section, we will deal with separable uniformly convex Banach spaces and will see
that also in this case convergence results can be achieved based on demicloseness considerations.

	An operator $T:X\to X$ is \emph{asymptotic regular at} $x\in X$ 
	if and only if 
	$$\lim_{n \rightarrow \infty} \|T^{n+1}x-T^n x\|=0,$$
	and it is said to be \emph{asymptotic regular} on $X$ if this holds true for every $x\in X$.

\begin{lemma}	\label{l:quasi-asym}
Let $(X,\|\cdot\|)$ be an $r$-uniformly convex Banach space and 
$T:X\to X$  a quasi $\alpha$-firmly nonexpansive operator. Then $T$ is asymptotic regular.
\end{lemma}

\begin{proof}
	By quasi $\alpha$-firmly nonexpansiveness of $T$ we obtain 
	$$\|T^{n+1}x-y\|\leq \|T^{n}x-y\|$$
	for all $n\in\mathbb{N}$ and all $y\in \Fix T$.
	This means that $(\|T^{n}x-y\|)_{n\in\mathbb{N}}$ 
	is a monotone decreasing non-negative sequence.  
	Hence $\lim_{n \rightarrow \infty}\|T^{n}x-y\|=d(y)$ 
	for some real number $d(y)$ (possibly depending on $y$) for every $y\in\Fix T$. 
	Again by quasi $\alpha$-firmly nonexpansiveness we obtain 
	$$\|T^{n+1}x-T^nx\|^r
	\leq 
	\frac{2\alpha c_r}{c(1-\alpha)}\Big(\|T^{n}x-y\|^r-\|T^{n+1}x-y\|^r
	\Big).$$
	Passing in the limit as $n\rightarrow +\infty$ gives the result.
\end{proof}

Let $(X^*,\|\cdot\|_*)$ denote the dual space of $(X,\|\cdot\|)$. 
A sequence $(x_n)_{n \in \mathbb N} \subseteq X$ \emph{converges weakly} to an element $x\in X$, 
denoted by $x_n\overset{w}\to x$ if  $\lim_{n \rightarrow \infty} f(x_n)=f(x)$ for all $f\in X^*$. 
An element $x\in X$ is a \emph{weak cluster point} of a sequence $(x_n)\subseteq X$ 
if and only if there is a subsequence $(x_{n_k})$ of $(x_n)$ such that $x_{n_k}\overset{w}\to x$.  
 A sequence $(x_n)_{n \in \N} \subseteq X$ is \emph{Fej\'er monotone} with respect to a set $S\subseteq X$, if
\begin{equation}	\label{eq:fejer}
	\|x_{n+1}-y\| \leq \|x_{n}-y\|
\end{equation}
for all $y\in S$ and all $n\in\mathbb{N}$.

A  Banach space $X$ is said to satisfy \emph{Opial's property} 
if and only if $x_n\overset{w}\to x$ implies
$$\liminf_{n \rightarrow \infty} \|x_n-x\|<\liminf_{n \rightarrow \infty}\|x_n-y\|$$ 
for all $y\in X\setminus\{x\}$. 
For $x_n\overset{w}\to x$ and $y\in X\setminus\{x\}$, Opial's  property implies
\begin{equation} 		\label{eq:equiv-Opial}
\limsup_{n \rightarrow \infty} \|x_n-x\|
=\lim_{k\rightarrow \infty} \|x_{n_k}-x\|
< 
\liminf_{k\rightarrow \infty}\|x_{n_k}-y\|
\leq 
\limsup_{n \rightarrow \infty}\|x_n-y\|.
\end{equation}
Not all Banach spaces enjoy Opial's property. 
For example Hilbert spaces and $\ell_p$ spaces for $p \in (1,\infty]$ 
satisfy the property, while $L^p((0,2\pi))$, $p \in (1,\infty] \setminus \{2\}$ does not. 

\begin{remark}
Let $X,\| \cdot\|)$ be a uniformly convex Banach space.
If $X$ has a weakly continuous duality mapping, then $X$ satisfies Opial's property \cite[Lemma 3]{Opial}.
Further, Opial's property is equivalent 
to the coincidence of the weak convergence with the so-called $\Delta$-convergence \footnote{$\Delta$-convergence is a notion of weak convergence for metric spaces introduced by \cite[Lim 1976]{Lim}.}, \cite[Theorem 3.19]{Solemini}. Moreover for every separable Banach space, 
there is  an equivalent norm such that Opial's property is satisfied \cite[Theorem 1]{Dulst}.
\end{remark}

In Theorem \ref{thm:opial}, we will recall a classical result of Opial \cite{Opial} 
on the convergence of fixed point iterations in spaces having Opial's property.
Indeed, Opial proved his result for Hilbert spaces based on \v Smulian's 
and mentioned that it can be generalized to uniformly convex Banach spaces.
Here we will give another proof of the theorem, which makes use of the following lemma.

\begin{lemma}	\label{l:weakcluster}
	Let $(X,\|\cdot\|)$ be a uniformly convex Banach space 	with Opial's property
	and
	 $(x_n)_{n \in \N} \subseteq X$ a Fej\'er monotone sequence with respect to a set $S\subseteq X$. 
	If all weak cluster points of  $(x_n)_{n \in \N}$ belong to $S$, then 
	$x_n\overset{w}\to x$ for some $x\in S$.
\end{lemma}

\begin{proof}
Since $(x_n)_{n \in \N}$ is Fej\'er monotone with respect to $S$, it is a bounded sequence. 
Hence it has a weakly convergent subsequence. 
First, we prove that $(x_n)_{n \in \N}$ can have at most one  weak cluster point in $S$.
Suppose in contrary, that there exist
two subsequences $(x_{n_k})_{k \in \N}$ and $(x_{m_k})_{k \in \N}$ 
such that $x_{n_k}\overset{w}\to x$ and $x_{m_k}\overset{w}\to y$ 
for some $x\neq y$ in $S$. 
Let $r_1 :=\limsup_{k \rightarrow \infty} \|x_{n_k}-x\|$ and $r_2:=\limsup_{k \rightarrow \infty}\|x_{m_k}-y\|$, where
w.l.o.g. 
$r_1 \leq r_2$. 
By \eqref{eq:equiv-Opial}, 
we have the inequality
\begin{align}\label{eq:inequalityOpial}
& r_2 <\limsup_{k \rightarrow \infty} \|x_{m_k}-x\|.
\end{align}
For every $\varepsilon>0$, there exists $k_0$ such that $\|x_{n_k}-x\|<r_1+\varepsilon$ for all 
$k\geq k_0$. 
By Fej\'er monotonicity with respect to $S$ we obtain that 
$\|x_{m_k}-x\| \le \|x_{n_k}-x\| < r_1 +\varepsilon$ whenever $m_k\geq n_{k_0}$. 
Consequently, there exists $k_1>0$ 
such that $\|x_{m_k}-x\|<r_2+\varepsilon$ for all
$k \geq k_1$. 
However this contradicts \eqref{eq:inequalityOpial}. 

Second, we show that the whole sequence $x_n\overset{w}\to x$ 
for some $x\in S$. Let $x\in S$ be the unique 
weak cluster point of the sequence $(x_n)_{n \in \N}$. 
If the whole sequence does not weakly converge to $x$,
then there is a weakly open neighborhood $U$ of $x$ 
such that $X\setminus U$ contains infinitely many terms of the sequence $(x_n)_{n \in \N}$. 
On the other hand, $\{x_n\,:\,x_n\in X\setminus U\}$ is bounded, 
so that it has a weakly convergent subsequence. 
Let $y$ be its weak cluster point. 
Since $X\setminus U$ is weakly closed and hence weakly sequentially closed, 
we know that $y\in X\setminus U$. By construction $y\neq x$. 
which contradicts the uniqueness of the weak cluster point.
\end{proof}

Here is Opial's theorem \cite{Opial} together with an alternative proof.

\begin{thm} \label{thm:opial}
	Let $(X,\|\cdot\|)$ be a uniformly convex Banach space satisfying Opial's property. 
	Let $T:X\to X$ be a nonexpansive operator. 
	If $\Fix T\neq\emptyset$ and $T$ is asymptotic regular, then for any $x_0\in X$, 
	the iterates $x_{n+1}:=Tx_{n}, n\in\mathbb{N}$ converge weakly to an element $x^*\in\Fix T$.  
\end{thm}

\begin{proof}
Since by assumption $\Fix T\neq\emptyset$, the operator $T$ is quasi nonexpansive.
For an arbitrary fixed $x_0\in X$, we consider the iterations 
$x_{n+1}:=Tx_n$, $n\in\mathbb{N}$. 
Then the sequence $(x_n)_{n\in\mathbb{N}}$ is F\'ejer monotone 
with respect to $\Fix T$ and therefore  bounded. 
Hence there exists a subsequence $(x_{n_k})_{k\in\mathbb{N}}$ 
which converges weakly to some $x^*\in X$. 
By the triangle inequality and since $T$ is  nonexpansive, we obtain 
$$\|x_{n_k}-Tx^*\|\leq \|x_{n_k}-Tx_{n_k}\|+\|Tx_{n_k}-Tx^*\|\leq \|x_{n_k}-Tx_{n_k}\|+\|x_{n_k}-x^*\|.$$
Therefore, passing to limit inferior, we obtain since $T$ is asymptotic regular
\begin{equation}\label{eq:opial}
\liminf_{k \rightarrow \infty} \|x_{n_k}-Tx^*\|
\leq 
\lim_{k \rightarrow \infty} \|x_{n_k}-Tx_{n_k}\|
+\liminf_{k \rightarrow \infty} \|x_{n_k}-x^*\|=\liminf_{k \rightarrow \infty}\|x_{n_k}-x^*\|,
\end{equation}
and by Opial's property further $x^*=Tx^*$ i.e. $x^*\in\Fix T$. 
Now by same arguments if $(x_{n_m})_{m \in \mathbb N}$ 
is another subsequence of $(x_n)_{n \in \mathbb N}$ 
converging weakly to some element $y^*\in X$, then $y^*\in \Fix T$. 
This means that all weak cluster points of $(x_n)_{n \in \mathbb N}$ lie in $\Fix T$. 
By Lemma \ref{l:weakcluster}, it follows that $x^*=y^*$ 
and that the whole sequence $(x_n)_{n \in \mathbb N}$ weakly converges to $x^*\in \Fix T$. 
\end{proof}

Based on Opial's theorem, we have the following corollary for quasi $\alpha$-firmly nonexpansive operators.

\begin{corollary}	\label{th:fixedpoint}
	Let $(X,\|\cdot\|)$ be an $r$-uniformly convex Banach space satisfying Opial's property and let $T:X\to X$ be a nonexpansive operator. 
	If $T$ is quasi $\alpha$-firmly nonexpansive, then for any $x_0\in X$ the iterates $x_{n+1}:=Tx_n, n\in\mathbb{N}$ converge weakly to an element $x^*\in\Fix T$. 
Moreover, if $\bar{x}_n:=P_{\Fix T}x_n$ for all $n\in\mathbb{N}$, then $\lim_{n \rightarrow \infty} \bar{x}_n=x^*$.
\end{corollary}

\begin{proof}
The first assertion follows immediately by Theorem \ref{thm:opial} and since by Lemma \ref{l:quasi-asym} 
the operator $T$ is asymptotic regular on $X$. 

To show the second assertion,  notice that by Lemma \ref{l:fixedpointset}, 
the set $\Fix T$ is nonempty, closed and convex. Therefore by Proposition \ref{p:projection} we have that $\bar{x}_n:=P_{\Fix T}x_n$ for all $n\in\mathbb{N}$ is well defined and unique. 
Moreover for all $m,n\in\mathbb{N}$ the following inequalities hold
\begin{equation}
\label{eq:projidentityseq}
\|x_n-\bar{x}_n\|^r+\frac{c_r}{2}\|\bar{x}_n-\bar{x}_m\|^r\leq \|x_n-\bar{x}_m\|^r.
\end{equation}
Then F\'ejer monotonicity of $(x_n)_{n \in \mathbb N}$ with respect to $\Fix T$
and \eqref{eq:projidentityseq} yield for $n\geq m$,
$$\frac{c_r}{2}\|\bar{x}_n-\bar{x}_m\|^r\leq \|x_m-\bar{x}_m\|^r-\|x_n-\bar{x}_n\|^r.$$
Passing in the limit as $m,n\rightarrow+\infty$ implies $\lim_{m,n}\|\bar{x}_n-\bar{x}_m\|=0$. 
Hence, $(\bar{x}_n)_{n \in \mathbb N}$ is a Cauchy sequence in $\Fix T$. 
Since $\Fix T$ is a closed set and hence complete, we conclude 
$\lim_n\bar{x}_n=\bar{x}^*$ for some $\bar{x}^*\in\Fix T$. 
Again, by Proposition \ref{p:projection}, we obtain
\begin{equation}
\label{eq:projidentity}
\|x_n-\bar{x}_n\|^r+\frac{c_r}{2}\|\bar{x}_n-x^*\|^r\leq \|x_n-x^*\|^r
\end{equation}
for all $n\in\mathbb{N}$.
Passing to the limit inferior implies
$$\liminf_{n \rightarrow \infty}\|x_n-\bar{x}^*\|^r+\frac{c_r}{2}\|\bar{x}^*-x^*\|^r
\leq
\liminf_{n \rightarrow \infty}\|x_n-x^*\|^r.$$
Finally, it follows by Opial's property that $\bar{x}^*=x^*$. This completes the proof.
\end{proof}
                                                                                             
To state our next convergence result, we need the notation of differentiability of a norm and demiclosedness of sets.
Recall that $X$ has a Fr\'echet differentiable norm 
$\|\cdot\|$, if  the norm as a function is Fr\'echet differentiable except for $x = 0$.
This is equivalent with the property that
for every $x$ on the unit sphere $S(X)$,
the limit
$\lim_{t\to 0}(\|x+ty\|-\|x\|)/t$ 
exists and is attained uniformly in $y\in S(X)$. Examples of Banach spaces with  Fr\'echet differentiable norm
are the spaces $L_p$, $p \in (1,\infty)$, see \cite[Theorem 8]{Sundr1967}.

A mapping $R:C\subseteq X\to X$ is \emph{demiclosed at} $y\in X$,
if $x_n\overset{w}\to x\in C$ and $Rx_n\to y$ as $n \rightarrow \infty$
implies $Rx=y$. 
If $R$ is demiclosed at every $y\in X$, we say that $R$ is \emph{demiclosed}.
The following theorem states a well-known result of Browder. 

\begin{thm}[Browder's demiclosedness principle \cite{Browder-demi}]	\label{th:browder}
Let $(X,\|\cdot\|)$ be a uniformly convex Banach space and $C\subseteq X$ 
a bounded, closed, convex set. If an operator $T:C\to X$ is nonexpansive, then $\Id-T$ is demiclosed. 
\end{thm}

As an immediate application of Browder's demiclosedness principle  we obtain:

\begin{thm}	\label{th:separable}
	Let $(X,\|\cdot\|)$ be a uniformly convex Banach space and 
	$T:X\to X$ a nonexpansive operator. 
	If $\Fix T\neq\emptyset$ and $T$ is asymptotically regular, 
	then, for any $x_0\in X$, the weak cluster points of the iterates $x_{n+1}:=Tx_{n}$, $n\in\mathbb{N}$ 
	belong to $\Fix T$. 
	If additionally $\|\cdot\|$ is Fr\'echet differentiable, 
	then these iterates converge weakly to a certain element in $\Fix T$.
\end{thm}

\begin{proof}
For any $y\in \Fix T$, we get by the nonexpansivity of $T$ 
that $\|Tx_{n+1}-y\|\leq \|Tx_n-y\|$, $n\in\mathbb{N}$. 
Especially, the sequence of iterates $(x_n)_{n\in\mathbb{N}}$ is bounded. 
Then it has a subsequence $(x_{n_k})_{k\in\mathbb{N}}$ 
which converges weakly to a certain element $x^*\in X$. 
Let $\mathbb{B}(x^*,R)$ be the closed ball of radius $R>0$ around $x^*$. 
For sufficiently large $R$, we have $x_{n}\in\mathbb{B}(x^*,R)$ for all $n\in\mathbb{N}$ 
and in particular $(x_{n_k})_{k\in\mathbb{N}} \subset \mathbb{B}(x^*,R)$. 
Let 
$\widetilde{T}:=T|_{\mathbb{B}(x^*,R)}$ 
be the restriction of $T$ on $\mathbb{B}(x^*,R)$.
Then $\widetilde{T}:\mathbb{B}(x^*,R)\to X$ is again a nonexpansive mapping. 
Since $(X,\|\cdot\|)$ is uniformly convex and $\mathbb{B}(x^*,R)$ is a bounded, closed, convex set, 
an application of Browder's demiclosedness principle \ref{th:browder}
implies that $I-\widetilde{T}$ is demiclosed. 
Note that by the asymptotic regularity of $T$, 
we get  $\|(I-\widetilde{T})x_{n_k}\|=\|x_{n_k}-Tx_{n_k}\|\to 0$ as $k\rightarrow+\infty$. By the demiclosedness of $\Id- \widetilde{T}$, 
we obtain $(\Id-\widetilde{T})x^*=0$ or equivalently 
$x^*=\widetilde{T}x^*=Tx^*$. 
Therefore $x^*\in\Fix T$. 
By the same arguments, if $(x_{n_m})_{m\in\mathbb{N}}$ 
is another weakly convergent subsequence, then its weak limit 
lies in $\Fix T$. Thus, all weak cluster points of the original sequence $(x_n)_{n\in\mathbb{N}}$ are in $\Fix T$. 

Now assume that $\|\cdot\|$ is Fr\'echet differentiable. 
Let $J:X\to X^*$ denote the normalized duality mapping defined by $$J(x):=\{f\in X^*\;:\;\langle x,f\rangle=\|x\|^2=\|f\|_*^2\},$$
where $\langle\cdot,\cdot\rangle$ is the dual pairing between $X$ and $X^*$. 
By virtue of \cite[Lemma 2.3]{tan-xu}, we know that 
$\lim_{n\to\infty}\langle x_n, J(u-v)\rangle$ exists for all $u,v\in\Fix T$ and in particular 
$\langle x-y,J(u-v)\rangle=0$, whenever $x, y$ are weak cluster points of $(x_n)_{n\in\mathbb{N}}$. 
Setting $u=x,v=y$, we obtain $\|x-y\|^2=\langle x-y,J(x-y)\rangle=0$ i.e. $x=y$. 
Since the weak cluster points $x$ and $y$ were arbitrary chosen, they all coincide 
with some $x^*\in\Fix T$. 
An application of the same argument as in Lemma \ref{l:weakcluster} yields that $x_n\overset{w}\to x^*\in\Fix T$.
\end{proof}

As a consequence of the last theorem 
we get the following result on nonexpansive mappings that are quasi $\alpha$-firmly nonexpansive.

\begin{corollary}
	\label{th:separable-p}
	Let $(X,\|\cdot\|)$ be an $r$-uniformly convex Banach space 
	and let $T:X\to X$ be a nonexpansive operator. If $T$ is quasi $\alpha$-firmly nonexpansive, 
	then, for any $x_0\in X$, the weak cluster points of the iterates $x_{n+1}:=Tx_{n}, n\in\mathbb{N}$ belong $\in \Fix T$. If additionally $\|\cdot\|$ is Fr\'echet differentiable then these iterates converge weakly to a certain element in $\Fix T$.
\end{corollary}

\begin{proof}
	Since $T$ is quasi $\alpha$-firmly nonexpansive operator, we have $\Fix T\neq\emptyset$,
	and $T$ is asymptotic regular by Lemma \ref{l:quasi-asym}. On the other hand, an $r$-uniformly convex Banach space is uniformly convex, so that conclusion follows from Theorem \ref{th:separable}.
\end{proof}

\section{Illustrative Examples} \label{sec:illustrations}
In this section, we illustrate by three examples, where the theory of $\alpha$-firmly nonexpansive operators
might be of interest.

\subsection{Deep Learning}   
Recently, neural networks on infinite dimensional spaces have received a certain attention \cite{CP2020}.
For applications of their finite dimensional counterparts, we refer to \cite{HHNPSS2019,HNS2021}.
To this end, let $(X,\|\cdot\|)$ be a real $\ell_p$-space, $p \in (1,\infty)$ and 
$A_{k}:X\to X$, $k=1,2,...,d$ a family of affine mappings which are $\alpha_k$-firmly nonexpansive.
Consider a so-called \textit{stable activation function} $\sigma:\mathbb R \to \mathbb R$ which acts elementwise  on the elements  $x\in X$.
Recall that Combettes and Pesquet \cite{Comb-Pesquet} called an activation function \emph{stable},
if it is increasing, $1$--Lipschitz continuous and $\sigma(0)=0$. 
The first two properties are equivalent to the fact that
$\sigma$ is $\frac12$-averaged. Note that most of the common activation functions are indeed stable.
Clearly, then $\sigma:X \rightarrow X$ (meant componentwise) is also $\frac12$-averaged, since
there is a nonexpansive operator $R:\R \rightarrow \R$ such that
$$
\sigma(x) = \left( \frac12(x_i + R x_i) \right)_{i \in \mathbb N} 
= \frac12 \left( x + (R x_i)_{i \in \mathbb N} \right).
$$
By Proposition \ref{prop:relation}, the activation function is also $\frac12$-firmly nonexpansive on $X$.
Then the neural network of depth $d\geq 1$ given by
\begin{eqnarray} \label{eq:neuralnetwork}
\Phi(x;(A_k)_{k=1}^d, \sigma):=A_d(\sigma(A_{d-1}...A_2(\sigma A_1x)))
\end{eqnarray}
is the composition of $\alpha$-firmly nonexpansive operators. 
Consequently, by Proposition \ref{p:convexcombinations} ii), 
the network itself is an $\alpha$-firmly nonexpansive operator.

\subsection{Semigroup Theory}
	Let $(X,\|\cdot\|)$ be a Banach space and $F:X\to X$ a nonexpansive operator. 
	Given $x\in X$ and $\lambda\in(0,+\infty)$, we define the mapping 
	\begin{equation} 	\label{eq:contraction-mapping}
	G_{x,\lambda}:y\mapsto \frac{1}{1+\lambda}x+\frac{\lambda}{1+\lambda}Fy,\quad y\in X.
	\end{equation}
	Then $G_{x,\lambda}$ is a contraction  with Lipschitz constant $\lambda/(1+\lambda)$. By Banach's fixed point theorem $G_{x,\lambda}$ has a unique fixed point, which we denote by $R_{\lambda}x$. 
	The mapping $x\mapsto R_{\lambda}x$ is called the \emph{resolvent} of $F$. 
	By convention $R_0x:=x$ for all $x\in X$.
	
	\begin{proposition} \label{prop:resolvent}
	Let $(X,\|\cdot\|)$ be an  $r$-uniformly convex Bancah space.
	Then the resolvent $R_{\lambda}$ of a nonexpansive operator $F:X\to X$ is $\alpha$-firmly nonexpansive 
	for any $\lambda>0$ and all $\alpha \geq 1/2$. 
	\end{proposition}
	
	\begin{proof}
	First, we have that $R_{\lambda}$ is nonexpansive, since for any $x,y\in X$ it follows by the nonexpansivity of $F$ that 
	\begin{align*}
	\|R_{\lambda}x-R_{\lambda}y\|
	&\leq \frac{1}{1+\lambda}\|x-y\|+\frac{\lambda}{1+\lambda}\|FR_{\lambda}x-FR_{\lambda}y\|\\
	&\leq \frac{1}{1+\lambda}\|x-y\|+\frac{\lambda}{1+\lambda}\|R_{\lambda}x-R_{\lambda}y\|,
	\end{align*}
	and a rearrangement of terms yields 
	$\|R_{\lambda}x-R_{\lambda}y\|\leq \|x-y\|$. 
	
	Now set 
	$u:=(1-w)x+wR_{\lambda}x$ 
	and 
	$v:=(1-w)y+wR_{\lambda}y$, 
	where 
	$w\in[0,1]$. 
	Let $s:=(\lambda-w\lambda)/(1+\lambda-w\lambda)$. 
	Then straightforward calculations show that $R_{\lambda}x=(1-s)u+sFR_{\lambda}x$ 
	and $R_{\lambda}y=(1-s)v+sFR_{\lambda}y$ and consequently
		$$
		\|R_{\lambda}x-R_{\lambda}y\|\leq (1-s)\|u-v\|+s\|FR_{\lambda}x-FR_{\lambda}y\|
		\leq 
		(1-s)\|u-v\|+s\|R_{\lambda}x-R_{\lambda}y\|.
		$$
		Rearranging terms yield 
		$\|R_{\lambda}x-R_{\lambda}y\|\leq \|u-v\|$. 
		By definition of $u$ and $v$, this means that $R_{\lambda}$ is firmly nonexpansive in the sense of Bruck, 
		see Remark \ref{rem:bruck}, 
		and hence, by Proposition \ref{p:Bruck},
		it follows that $R_{\lambda}$ is $\alpha$-firmly nonexpansive for any 
		$\alpha \geq 1/2$. 
    \end{proof}
		
		Proposition \ref{prop:resolvent} has two important implications. 
		First, we get inclined in studying the fixed point problem $R_{\lambda}x=x$ instead of $Fx=x$. 
		This is further supported by the fact that $\Fix F=\Fix R_{\lambda}$. 
		It is trivial by definition of $R_{\lambda}$ to notice that $x\in\Fix R_{\lambda}$ implies $x\in \Fix F$. 
		For the other direction i.e. when $x\in \Fix F$,  we get from the following chain of equalities and inequalities
		$$0\leq \|R_{\lambda}x-x\|=\frac{\lambda}{1+\lambda}\|FR_{\lambda}x-x\|=\frac{\lambda}{1+\lambda}\|FR_{\lambda}x-Fx\|\leq \frac{\lambda}{1+\lambda}\|R_{\lambda}x-x\|$$
		that $x=R_{\lambda}x$.
		When the space $X$ is $r$-uniformly convex satisfying Opial's property (else a separable space),
		then by Theorem \ref{th:fixedpoint} (Theorem \ref{th:separable-p}), 
		we obtain that for any initialization $x_0\in X$, 
		the iterates $x_{n+1}:=R_{\lambda}x_n$ converge weakly to an element in $\Fix R_{\lambda}$,
		and therefore to an element in $\Fix F$. 
		
		Second, we will demonstrate that there is an intimate relationship between the class of $\alpha$-firmly nonexpansive operators and the theory of strongly continuous semigroups. 
		Let $t>0$ be fixed and choose $n\in \mathbb{N}$. 
		Consider the operator 
		$$T_{n,t}:=\underbrace{R_{\frac{t}{n}}\circ R_{\frac{t}{n}}\circ...\circ R_{\frac{t}{n}}}_{n-\text{times}}.$$ 
		By virtue of Proposition \ref{p:avgcvx} ii)
		it follows that $T_{n,t}$ is itself $\alpha$-firmly nonexpansive with constant $\alpha_n=n^{p-1}/(n^{p-1}+1)$. 
		If $T_tx:=\lim_{n \rightarrow \infty} T_{n,t}x$ for all $x\in X$, 
		then it is evident that $T_t$ 
		is nonexpansive, whenever this limit exists. 
		It can however been shown that such limit always exists 
		and it is uniform in $t$ on bounded intervals, see \cite{Liggett} or \cite[Theorem 4.3.3]{Bacak}). 
		Moreover, the family of operators $(T_t)_{t>0}$ defines a strongly continuous semigroup of nonexpansive operators, i.e.,
		i) $\lim_{t\rightarrow 0}T_tx=x$, 
		ii)  $T_s(T_tx)=T_{s+t}x$ for every $s,t\geq 0$, and 
		iii)  $\|T_tx-T_ty\|\leq \|x-y\|$ for all $x,y\in X$ and $t\geq 0$.
		
\subsection{Contractive Projections in $L_p$ Spaces} 
Let $(X,\|\cdot\|)$ by a Banach space.
An operator $P:X\to X$ is a \emph{contractive projection},
if it is a linear operator, $P^2=P$ and $\|P\|\leq 1$. 
For Lebesgue spaces $X=L_p$, $p \in (1,\infty) \backslash \{2\}$
on a probability measure space $(\Omega,\Sigma, \mu)$,
Ando \cite{Ando} proved that if $P:L_p\to L_p$ is a contractive projection such that $P(1)=1$, 
then $P$ is a conditional expectation $\mathbb{E}_{\mathscr{B}}$ 
with respect to some sub-$\sigma$-algebra $\mathscr{B}$ of $\Sigma$.  
In general, every contractive projection $P$ on $L_p$ induces a canonical conditional expectation. 
Byrne and Sullivan investigated such operators 
with the additional property that $\Id-P$ is also contractive and showed the following result.

\begin{thm}[Structure of contractive projections \cite{Byrne}]\label{t:byrne}
Let $(X,\|\cdot\|)$ be a Lebesgue space on a probability measure space $(\Omega,\Sigma,\mu)$.
Then the operators $P: X \rightarrow X$ and $\Id-P$ are contractive 
if and only if $P=(\Id+U)/2$ for some isometry $U: X \rightarrow X$ satisfying $U^2=\Id$. 
\end{thm}

This elementary observation puts the theory of conditional expectations in $L_p$ spaces in direct relation 
with the theory of averaged operators and consequently with $\alpha$-firmly nonexpansive operators. 
In particular, we obtain that when $P$ and $\Id-P$ are both contractive, then $P$ and $\Id-P$ 
are $\alpha$-firmly nonexpansive for any $\alpha\geq 1/2$. 

Having contractive operators at hand, it is then desirable to investigate the so-called 
\emph{feasibility problem in Banach spaces}: 
Given a Banach space $X$ and a finite number of subspaces $\mathscr{S}_i\subseteq X$ for $i=1,2,...,n$, find an element $x\in\bigcap_{i=1}^n\mathscr{S}_i$, provided that $\bigcap_{i=1}^n\mathscr{S}_i$ is nonempty. 

A popular technique for solving such a problem, at least in the setting of a Hilbert space, is the method of alternating projections. Here an arbitrary point is chosen which is then projected to some subspace by means of the metric (orthogonal) projection operator, then this projection is projected onto the next subspace and so on it repeats itself in a cyclic order. 
It is known that such a generated sequence always converges weakly to a certain element in the intersection of the subspaces. 
The method of alternating projections was first considered by von Neumann \cite{vNeumann} 
for the case of two intersecting subspaces in a Hilbert space, and then generalized by Halpern \cite{Halpern} to an arbitrary finite number of intersecting subspaces. 
While the theory of metric projections works well in Hilbert spaces, 
this is no longer true for general Banach spaces since metric projections 
are not nonexpansive, see Remark \ref{rem:hilbert}. 
However, we can still solve the feasibility problem, 
at least for Lebesgue spaces which are either separable or satisfy Opial's property, 
by using a method which we will call the \emph{method of alternating contractive projections}. 
Here, for given subspaces $\mathscr{S}_i$, $i=1,2,...,n$,
we require an equal number of contractive projection operators $P_i$ such that 
$P_i(X)=\mathscr{S}_i$, $i=1,2,...,n$. 
However, in a Lebesgue space not every subspace is the image of some contractive projection. 
While it is known that every subspace can be the image of at most one contractive projection, 
only those subspaces which are isometric to an $L_p$ space over some measure space $(\Omega,\Sigma,\mu)$ admit such a property \cite[Theorem 4]{Ando}. 
Such subspaces are known as $L_p$-type subspaces.

Recall that Lebesgue spaces are examples of $r$-uniformly convex Banach spaces with Fr\'echet differentiable norm. Here we restrict ourselves to feasibility problems in Lebesgue spaces $L_p, p\in(1,\infty)\setminus\{2\}$ on a probability measure space $(\Omega,\Sigma,\mu)$, 
where the intersecting subspaces are the images of certain contractive projections. 
We then can show the following result.

\begin{thm}
	\label{th:alternating}
	Let $(X, \|\cdot\|)$ be a Lebesgue space and $\mathscr{S}_i\subseteq X$ be closed, linear subspaces of $L_p$-type 
	that have nonempty intersection. 
	Let $P_i:X\to X$ be contractive projections with contractive complement $\Id-P_i$, 
	such that $P_i(X)=\mathscr{S}_i$, $i=1,2,...,n$. 
	Set $P\coloneqq P_nP_{n-1}...P_2P_1$.
	Then, for any given $x_0\in X$  the iterates $x_{n}:=Px_{n-1}$, $n\in\mathbb{N}$ 
    converge weakly to an element $x^*\in\bigcap_{i=1}^n\mathscr{S}_i$.
\end{thm}

\begin{proof}
Since for every contractive projection $P_i$ its complement $\Id-P_i$ is itself contractive, there exists by Theorem
\ref{t:byrne} an isometry $U_i:X\to X$ such that $U_i^2=\Id$ and $P_i=(\Id+U_i)/2$. 
In particular, $P_i$ is $\alpha$-firmly nonexpansive for every $i=1,2,...,n$ and $\alpha\geq1/2$. 

Further, we see by the following reasons that $\Fix P_i=\mathscr{S}_i$, $i=1,2,...,n$:
Since $P_i(X)=\mathscr{S}_i$, there exists for any  $y\in \mathscr{S}_i$ an $x\in X$ such that $P_ix=y$. 
By definition $P_i$ is idempotent, so that 
$y=P_ix=P^2_ix=P_iy$. This implies that $y\in\Fix P_i$, and thus $\mathscr{S}_i\subseteq\Fix P_i$. 
For the other direction, note that 
$x=P_ix\in P_i(X)=\mathscr{S}_i$. 
Therefore, $\Fix P_i=\mathscr{S}_i$. 

Then, the assumption $\bigcap_{i=1}^n\mathscr{S}_i\neq\emptyset$ 
is equivalent to $\bigcap_{i=1}^n\Fix P_i\neq\emptyset$. 
In particular, we have $\Fix P_i\neq\emptyset$ for every $i=1,2,...,n$, so that 
$P_i$ is also a quasi $\alpha$-firmly nonexpansive operator. 
By Proposition \ref{p:calculus-quasi}, 
the composition $P:=P_nP_{n-1}...P_2P_1$ is quasi $\alpha$-firmly nonexpansive 
and by Proposition \ref{p:fixedpointset-compositions} 
we have $\Fix P=\bigcap_{i=1}^n\Fix P_i=\bigcap_{i=1}^n\mathscr{S}_i$. 
Moreover, $P$ is nonexpansive operator as a composition of finitely many such operators. 
Since $X$ is a Lebesgue space and in particular an $r$-uniformly convex space with Fr\'echet differentiable norm, then Corollary \ref{th:separable-p} 
implies 
that the iterates $x_{n}:=Px_{n-1}$, $n\in\mathbb{N}$ converge weakly to an element $x^*\in\bigcap_{i=1}^n\mathscr{S}_i$.
\end{proof}

For an illustrative example consider the $\ell_p$ space. Notice that $\ell_p$ is a Lebesgue space $L_p$ on a measure space $(\Omega,\Sigma,\mu)$ where $\Omega:=\mathbb{N}, \Sigma:=2^{\mathbb{N}}$ and $\mu$ is the usual counting measure.  Furthermore let $U, V:\ell_p\to\ell_p$ be operators acting by the formulae $Ux:=(x_2,x_1,x_3,x_4,...)$ and $Vx:=(x_1,x_3,x_2,x_4,...)$ for a given $x\in \ell_p$ where $x:=(x_1,x_2,x_3,...)$. Evidently both $U$ and $V$ are isometries and satisfy $U^2=V^2=\Id$. If $P_U:=(\Id+U)/2$ and $P_V:=(\Id+V)/2$ then in view of Theorem \ref{t:byrne} the operators $P_U, \Id-P_U$ and $P_V, \Id-P_V$ are contractive projections. Moreover let $\mathscr{S}_U$ be the set of elements in $\ell_p$ of the form $(a,a,\ast,\ast,...)$ and $\mathscr{S}_V$ be the set of elements $(\ast,a,a,\ast,...)$ where $a\in\mathbb{R}$. Clearly $\mathscr{S}_U,\mathscr{S}_V$ are closed linear subspaces of $\ell_p$ and are invariant under the isometries $U$ and $V$ respectively. In particular we have $P_U(\ell_p)=\mathscr{S}_U$ and $P_V(\ell_p)=\mathscr{S}_V$. Since $\mathscr{S}:=\mathscr{S}_U\cap\mathscr{S}_V$ is nonempty, in fact it consists of all elements of the form $(a,a,a,\ast,\ast,...)$, then by Theorem \ref{th:alternating} the iterates $x_n:=P_VP_Ux_{n-1}, n\in\mathbb{N}$ converge weakly to an element in $\mathscr{S}$.

Another method for solving feasibility problems in a Hilbert space 
is that of averaged projections. Let us consider 
this method in Lebesgue spaces.

\begin{thm}
	Let $(X, \|\cdot\|)$ be a Lebesgue space and $\mathscr{S}_i\subseteq X$ be closed linear subspaces of $L_p$-type that have nonempty intersection. 
	Let $P_i:X\to X$ be contractive projections with contractive complement $\Id-P_i$, such that 
	$P_i(X)=\mathscr{S}_i$, $i=1,2,...,n$. Set $P:= \sum_{i=1}^nw_iP_i$, where $w_i\in(0,1)$ fulfill
	$\sum_{i=1}^nw_i=1$. Then, for any given $x_0\in X$ the iterates $x_{n}:=Px_{n-1}$, $n\in\mathbb{N}$ converge weakly to an element $x^*\in\bigcap_{i=1}^n\mathscr{S}_i$
\end{thm}
\begin{proof}
	Follows similar arguments as in the last Theorem.
\end{proof}

\appendix
\section{Proof of Proposition \ref{p:avgcvx}}
\begin{proof}
i)	By  \eqref{eq:avg} there exist nonexpansive operators $R_i:X\to X$ 
such that $T_i:=(1-\alpha_i)\Id+\alpha_iR_i$, $i=1,2,...,n$. 
Define $\alpha:=\sum_{i=1}^nw_i\alpha_i$, then 
	$$
	T:=\sum_{i=1}^nw_iT_i=\sum_{i=1}^nw_i(1-\alpha_i)\Id+w_i\alpha_iR_i=(1-\alpha)\Id+\sum_{i=1}^nw_i\alpha_iR_i.
	$$ 
	It suffices to show that $R:=\sum_{i=1}^nw_i\alpha_iR_i/\alpha$ is nonexpansive. 
	This follows for all $x,y\in X$ from 
	$$
	\|Rx-Ry\|\leq\frac{1}{\alpha}\sum_{i=1}^nw_i\alpha_i\|R_ix-R_iy\|
	\leq\frac{1}{\alpha}\sum_{i=1}^nw_i\alpha_i\|x-y\|=\|x-y\|.
	$$
ii)	We prove the second claim by induction starting with two $\alpha$-averaged operators 
 $T_1,T_2:X\to X$  with averaging constants $\alpha_1,\alpha_2\in(0,1)$. 
Then 
	$$T:=T_2T_1=((1-\alpha_2)\Id+\alpha_2R_2)((1-\alpha_1)\Id+\alpha_1R_1)$$
	for some nonexpansive mappings $R_1,R_2:X\to X$ and
	we can rewrite $T$ as 
	\begin{align*}
	T&=(1-\alpha_1)(1-\alpha_2)\Id 
	+ \alpha_1(1-\alpha_2)R_1+\alpha_2R_2((1-\alpha_1)\Id+\alpha_1R_1)\\
	&=(1-\alpha)\Id + \alpha R,
	\end{align*}
	where $\alpha := 1- (1-\alpha_1)(1-\alpha_2)$ and
	$$
	R:=\frac{\alpha_1(1-\alpha_2)}{\alpha}R_1
	+\frac{\alpha_2}{\alpha }R_2((1-\alpha_1)\Id+\alpha_1R_1).
	$$
	For any $x,y\in X$ we have
	\begin{align*}
		\|Rx-Ry\|&\leq\frac{\alpha_1(1-\alpha_2)}{\alpha_1(1-\alpha_2)+\alpha_2}\|R_1x-R_1y\|\\
		&+\frac{\alpha_2}{\alpha_1(1-\alpha_2)+\alpha_2}\|R_2((1-\alpha_1)x+\alpha_1R_1x)- R_2((1-\alpha_1)y+\alpha_1R_1y)\|.
	\end{align*}
By nonexpansivity of $R_1, R_2$ and since $\| \cdot\|$ is convex, we obtain
\begin{align*}
	\|Rx-Ry\|
	&\leq \frac{\alpha_1(1-\alpha_2)}{\alpha}\|x-y\|\\&+\frac{\alpha_2}{\alpha}\|(1-\alpha_1)x
	+\alpha_1R_1x- ((1-\alpha_1)y+\alpha_1R_1y)\|\\
	&=
	\frac{\alpha_1(1-\alpha_2)}{\alpha}\|x-y\|+\frac{\alpha_2}{\alpha}\|(1-\alpha_1)(x-y)+\alpha_1(R_1x-R_1y)\|\\
	&\leq 
	\frac{\alpha_1(1-\alpha_2)}{\alpha}\|x-y\|+\frac{\alpha_2}{\alpha}\|x-y\|\\&=\|x-y\|,
\end{align*} 
so that $R$ is a nonexpansive operator. 

Let us assume the the claim is true for the composition of $n-1$ operators and consider
$T := T_n T_{n-1} \ldots T_1$. By assumption  $\tilde T := T_{n-1} \ldots T_1$ is an $\alpha$-averaged operator 
with constant $\tilde \alpha = 1 - \Pi_{i=1}^{n-1} (1-\alpha_i)$ and by the above considerations $T = T_n \tilde T$
is an averaged operator with constant $\alpha = 1 - (1-\alpha_n)(1-\tilde \alpha) = 1 - \Pi_{i=1}^{n} (1-\alpha_i)$.
\end{proof}

\paragraph{Acknowldegment:}
Funding by the DFG under Germany's Excellence Strategy – The Berlin Mathematics Research Center MATH+ (EXC-2046/1,  Projektnummer:  390685689) is acknowledged. Many thanks to Marzieh Hasannasab for fruitful discussions.

\bibliographystyle{plain}
\bibliography{fne_Berdellima--Steidl}


\end{document}